\documentclass[11pt]{amsart}

\usepackage{graphics}
\usepackage{color}
\usepackage[a4paper, margin=3cm]{geometry}
\usepackage{array}
\usepackage{amssymb,enumerate}
\usepackage{amsmath}
\usepackage{bbm}           
\usepackage{bm}
\usepackage{eso-pic,graphicx}
\usepackage{tikz}
\usepackage{cite}
\usepackage{esint}
\usepackage[pdftex, bookmarksnumbered, bookmarksopen, colorlinks, citecolor=blue, linkcolor=blue]{hyperref}

\usepackage{graphicx}
\usepackage{bbding}
\usepackage{makecell}
\usepackage{amsmath,amsfonts}
\usepackage{rotating}
\usepackage{amssymb}
\usepackage{verbatim}
\usepackage{rotating}
\usepackage{mathrsfs}
\DeclareSymbolFont{largesymbols}{OMX}{cmex}{m}{n}
\makeatletter
\def\Ddots{\mathinner{\mkern1mu\raise\p@
\vbox{\kern7\p@\hbox{.}}\mkern2mu
\raise4\p@\hbox{.}\mkern2mu\raise7\p@\hbox{.}\mkern1mu}}
\makeatother

\def\XXint#1#2#3{{\setbox0=\hbox{$#1{#2#3}{\int}$}
\vcenter{\hbox{$#2#3$}}\kern-.5\wd0}}

\begin{document}

\newtheorem{hyp}{Hypothesis}

\newtheorem{hyp2}[hyp]{Hypothesis}

\newtheorem{definition}{Definition}
\newtheorem{theorem}[definition]{Theorem}
\newtheorem{proposition}[definition]{Proposition}
\newtheorem{conjecture}[definition]{Conjecture}
\def\theconjecture{\unskip}
\newtheorem{corollary}[definition]{Corollary}
\newtheorem{lemma}[definition]{Lemma}
\newtheorem{claim}[definition]{Claim}
\newtheorem{sublemma}[definition]{Sublemma}
\newtheorem{observation}[definition]{Observation}
\theoremstyle{definition}

\newtheorem{notation}[definition]{Notation}
\newtheorem{remark}[definition]{Remark}
\newtheorem{question}[definition]{Question}

\newtheorem{example}[definition]{Example}
\newtheorem{problem}[definition]{Problem}
\newtheorem{exercise}[definition]{Exercise}
 \newtheorem{thm}{Theorem}
 \newtheorem{cor}[thm]{Corollary}
 \newtheorem{lem}{Lemma}[section]
 \newtheorem{prop}[thm]{Proposition}
 \theoremstyle{definition}
 \newtheorem{dfn}[thm]{Definition}
 \theoremstyle{remark}
 \newtheorem{rem}{Remark}
 \newtheorem{ex}{Example}
 \numberwithin{equation}{section}

\def\C{\mathbb{C}}
\def\R{\mathbb{R}}
\def\Rn{{\mathbb{R}^n}}
\def\Rns{{\mathbb{R}^{n+1}}}
\def\Sn{{{S}^{n-1}}}
\def\M{\mathbb{M}}
\def\N{\mathbb{N}}
\def\Q{{\mathbb{Q}}}
\def\Z{\mathbb{Z}}
\def\X{\mathbb{X}}
\def\Y{\mathbb{Y}}
\def\F{\mathcal{F}}
\def\L{\mathcal{L}}
\def\S{\mathcal{S}}
\def\supp{\operatorname{supp}}
\def\essi{\operatornamewithlimits{ess\,inf}}
\def\esss{\operatornamewithlimits{ess\,sup}}

\numberwithin{equation}{section}
\numberwithin{thm}{section}
\numberwithin{definition}{section}
\numberwithin{equation}{section}

\def\earrow{{\mathbf e}}
\def\rarrow{{\mathbf r}}
\def\uarrow{{\mathbf u}}
\def\varrow{{\mathbf V}}
\def\tpar{T_{\rm par}}
\def\apar{A_{\rm par}}

\def\reals{{\mathbb R}}
\def\torus{{\mathbb T}}
\def\t{{\mathcal T}}
\def\heis{{\mathbb H}}
\def\integers{{\mathbb Z}}
\def\z{{\mathbb Z}}
\def\naturals{{\mathbb N}}
\def\complex{{\mathbb C}\/}
\def\distance{\operatorname{distance}\,}
\def\support{\operatorname{support}\,}
\def\dist{\operatorname{dist}\,}
\def\Span{\operatorname{span}\,}
\def\degree{\operatorname{degree}\,}
\def\kernel{\operatorname{kernel}\,}
\def\dim{\operatorname{dim}\,}
\def\codim{\operatorname{codim}}
\def\trace{\operatorname{trace\,}}
\def\Span{\operatorname{span}\,}
\def\dimension{\operatorname{dimension}\,}
\def\codimension{\operatorname{codimension}\,}
\def\nullspace{\scriptk}
\def\kernel{\operatorname{Ker}}
\def\ZZ{ {\mathbb Z} }
\def\p{\partial}
\def\rp{{ ^{-1} }}
\def\Re{\operatorname{Re\,} }
\def\Im{\operatorname{Im\,} }
\def\ov{\overline}
\def\eps{\varepsilon}
\def\lt{L^2}
\def\diver{\operatorname{div}}
\def\curl{\operatorname{curl}}
\def\etta{\eta}
\newcommand{\norm}[1]{ \|  #1 \|}
\def\expect{\mathbb E}
\def\bull{$\bullet$\ }

\def\blue{\color{blue}}
\def\red{\color{red}}

\def\xone{x_1}
\def\xtwo{x_2}
\def\xq{x_2+x_1^2}
\newcommand{\abr}[1]{ \langle  #1 \rangle}

\newcommand{\Norm}[1]{ \left\|  #1 \right\| }
\newcommand{\set}[1]{ \left\{ #1 \right\} }
\newcommand{\ifou}{\raisebox{-1ex}{$\check{}$}}
\def\one{\mathbf 1}
\def\whole{\mathbf V}
\newcommand{\modulo}[2]{[#1]_{#2}}
\def \essinf{\mathop{\rm essinf}}
\def\scriptf{{\mathcal F}}
\def\scriptg{{\mathcal G}}
\def\m{{\mathcal M}}
\def\scriptb{{\mathcal B}}
\def\scriptc{{\mathcal C}}
\def\scriptt{{\mathcal T}}
\def\scripti{{\mathcal I}}
\def\scripte{{\mathcal E}}
\def\V{{\mathcal V}}
\def\scriptw{{\mathcal W}}
\def\scriptu{{\mathcal U}}
\def\scriptS{{\mathcal S}}
\def\scripta{{\mathcal A}}
\def\scriptr{{\mathcal R}}
\def\scripto{{\mathcal O}}
\def\scripth{{\mathcal H}}
\def\scriptd{{\mathcal D}}
\def\scriptl{{\mathcal L}}
\def\scriptn{{\mathcal N}}
\def\scriptp{{\mathcal P}}
\def\scriptk{{\mathcal K}}
\def\frakv{{\mathfrak V}}
\def\v{{\mathcal V}}
\def\C{\mathbb{C}}
\def\D{\mathcal{D}}
\def\R{\mathbb{R}}
\def\Rn{{\mathbb{R}^n}}
\def\rn{{\mathbb{R}^n}}
\def\Rm{{\mathbb{R}^{2n}}}
\def\r2n{{\mathbb{R}^{2n}}}
\def\Sn{{{S}^{n-1}}}
\def\bbM{\mathbb{M}}
\def\N{\mathbb{N}}
\def\Q{{\mathcal{Q}}}
\def\Z{\mathbb{Z}}
\def\F{\mathcal{F}}
\def\L{\mathcal{L}}
\def\G{\mathscr{G}}
\def\ch{\operatorname{ch}}
\def\supp{\operatorname{supp}}
\def\dist{\operatorname{dist}}
\def\essi{\operatornamewithlimits{ess\,inf}}
\def\esss{\operatornamewithlimits{ess\,sup}}
\def\dis{\displaystyle}
\def\dsum{\displaystyle\sum}
\def\dint{\displaystyle\int}
\def\dfrac{\displaystyle\frac}
\def\dsup{\displaystyle\sup}
\def\dlim{\displaystyle\lim}
\def\bom{\Omega}
\def\om{\omega}

\author[H. Yang]{Heng Yang}
\address{Heng Yang:
	College of Mathematics and System Sciences \\
	Xinjiang University \\
	Urumqi 830017 \\
	China}
\email{yanghengxju@yeah.net}

\author[J. Zhou]{Jiang Zhou$^{*}$}
\address{Jiang Zhou:
	College of Mathematics and System Sciences \\
	Xinjiang University \\
	Urumqi 830017 \\
	China}
\email{zhoujiang@xju.edu.cn}


\thanks{This work was supported by the National Natural Science Foundation of China (No.12061069).
\thanks{$^{*}$ Corresponding author, e-mail address: zhoujiang@xju.edu.cn}}

\date{July 5, 2023.}
\title[ Characterization of BMO spaces via commutators  ]
{\bf Characterization of BMO spaces via commutators of fractional maximal function on Lorentz spaces}

\begin{abstract}
Let $0 \leq \alpha<n$ and $b$ be the locally integrable function. In this paper, we consider the maximal commutator of fractional maximal function $M_{b,\alpha}$ and the nonlinear commutator of fractional maximal function $[b, M_{\alpha}]$ on Lorentz spaces.
We give some necessary and sufficient conditions for the boundedness of the commutators $M_{b,\alpha}$ and $[b, M_{\alpha}]$  on Lorentz spaces when the function $b$ belongs to BMO spaces, by which some new characterizations of certain subclasses of BMO spaces are obtained.
\end{abstract}

\maketitle

\section{Introduction and main results}\label{sec1}

Let $T$ be the classical singular integral operator and $b$ be the locally integrable function, the commutator $[b, T]$ is defined by
$$
[b,T]f(x)=bT f(x)-T (bf)(x).
$$
In 1976, Coifman, Rochberg and Weiss\cite{CRW} obtained that the commutator $[b,T]$ is bounded on $L^{p}(\mathbb{R}^{n})$ for $1<p<\infty$ if and only if $b \in BMO(\mathbb{R}^{n})$.
The bounded mean oscillation space $BMO(\mathbb{R}^{n})$ was introduced by John and Nirenberg \cite{JN}, which is defined as the set of all locally integrable functions $f$ on $\mathbb{R}^{n}$ such that
$$
\|f\|_{BMO(\mathbb{R}^{n})}:=\sup _{B} \frac{1}{|B|} \int_{B}|f(x)-f_{B}| d x<\infty,
$$
where the supremum is taken over all balls in $\mathbb{R}^{n}$ and  $f_{B}:=\frac{1}{|B|} \int_{B} f(x) d x$.
In 1978, Janson\cite{J} gave some characterizations of the Lipschitz space $\dot{\Lambda}_{\beta}\left(\mathbb{R}^{n}\right)$ via the commutator $[b, T]$ and proved that $[b, T]$ is bounded from $L^{p}\left(\mathbb{R}^{n}\right)$ to $L^{q}\left(\mathbb{R}^{n}\right)$ if and only if $b \in \dot{\Lambda}_{\beta}\left(\mathbb{R}^{n}\right)(0<$ $\beta<1)$, where $1<p<q<\infty$ and $\frac{1}{p}-\frac{1}{q}=\frac{\beta}{n}$.

As usual, $|B|$ is the Lebesgue measure of the ball $B$ and $\chi_{B}$ is the characteristic function of the ball $B$. For $1\leq p\leq\infty$, we denote by $p^{\prime}$ the conjugate index of $p$, namely, $p^{\prime}=\frac{p}{p-1}$. We always denote by  $C$  a positive constant which is independent of main parameters, but it may vary from line to line. The symbol  $f \lesssim g$  means that  $f \leq C g$. If  $f \lesssim g $ and  $g \lesssim f$, then we write  $f \approx g $.

Let $0 \leq \alpha<n$, for a locally integrable function $f$, the fractional maximal function $M_{\alpha}$ is given by
$$M_{\alpha}(f)(x)=\sup _{B \ni x} \frac{1}{|B|^{1-\frac{\alpha}{n}}} \int_{B}|f(y)| d y,$$
where the supremum is taken over all balls $B \subset \mathbb{R}^{n}$ containing $x$.

When $\alpha=0$, $M_{0}$ is the classical Hardy-Littlewood maximal function $M$, and $M_{\alpha}$ is the classical fractional maximal function when $0<\alpha<n$.

The maximal commutator of fractional maximal function $M_\alpha$ with the locally integrable function $b$ is given by
$$
M_{\alpha, b}(f)(x)=\sup _{B \ni x} \frac{1}{|B|^{1-\frac{\alpha}{n}}} \int_B|b(x)-b(y)||f(y)| d y,
$$
where the supremum is taken over all balls $B \subset \mathbb{R}^n$ containing $x$.

The  nonlinear commutator of fractional maximal function $M_\alpha$ with the locally integrable function $b$ is defined as
$$
\left[b, M_\alpha\right](f)(x)=b(x) M_\alpha(f)(x)-M_\alpha(b f)(x).
$$

When $\alpha=0$, we simply write by $[b, M]=\left[b, M_0\right]$ and $M_b=M_{0, b}$.

For a function $b$ defined on $\mathbb{R}^{n}$, we denote
$$
b^{-}(x):= \begin{cases}0, & \text { if } b(x) \geq 0. \\ |b(x)|, & \text { if } b(x)<0.\end{cases}
$$
and $b^{+}(x):=|b(x)|-b^{-}(x)$. Obviously, $b^{+}(x)-b^{-}(x)=b(x)$.

Let $b$ be any non-negative locally integrable function. Then, for any locally integrable function $f$ and $x \in \mathbb{R}^{n}$,
$$
\begin{aligned}
\left|\left[b, M_\alpha\right] f(x)\right|
& =\left|b(x) M_\alpha f(x)-M_\alpha(b f)(x)\right| \\
&=\Big|b(x)\sup _{B \ni x} \frac{1}{|B|^{1-\frac{\alpha}{n}}} \int_{B}|f(y)| d y-\sup _{B \ni x} \frac{1}{|B|^{1-\frac{\alpha}{n}}} \int_{B}|b(y)f(y)| d y\Big|\\
& \leq\sup _{B \ni x} \frac{1}{|B|^{1-\frac{\alpha}{n}}} \int_B|b(x)-b(y)||f(y)| d y \\
& =M_{b, \alpha}(f)(x) .
\end{aligned}
$$

If $b$ is any locally integrable function on $\mathbb{R}^{n}$. Then, for any locally integrable function $f$ and $x \in \mathbb{R}^{n}$,
\begin{equation} \label{5}
\left|\left[b, M_\alpha\right] f(x)\right| \leq M_{b, \alpha}(f)(x)+2 b^{-}(x) M_\alpha f(x) \tag{1.1}
\end{equation}
holds (see, for example, \cite{ZWS}).
Obviously, the commutators $M_{\alpha, b}$ and $\left[b, M_\alpha\right]$ essentially differ from each other. The
maximal commutator $M_{\alpha, b}$ is positive and sublinear, but nonlinear commutator $[b, M_\alpha]$ is neither positive nor sublinear.

For a fixed ball $B$ and  $0 \leq \alpha< n$, the fractional maximal function with respect to $B$ of a function $f$ is given by
$$
M_{\alpha,B}(f)(x)=\sup _{B \supseteq B_0 \ni x} \frac{1}{|B_0|^{1-\frac{\alpha}{n}}} \int_{B_0}|f(y)| dy,
$$
where the supremum is taken over all balls $B_0$ with $B_0 \subseteq B$ and $B_0 \ni x$. Moreover, we denote by $M_{B} = M_{0,B}$ when $\alpha= 0$.

We also need to recall the non-increasing rearrangement of a real function. Suppose that $f$ is a measurable function on $\mathbb{R}^n$. Then we define
$$
f^*(t)=\inf \left\{s>0: d_f(s) \leq t\right\},
$$
where
$$
d_f(s):=|\{x \in \mathbb{R}^n:|f(x)|>s\}|, \quad \forall s>0 .
$$

For $0<p<\infty$, the Lebesgue space  $L^{p}(\mathbb{R}^{n})$  is defined as the set of all measurable functions  $f$  on  $\mathbb{R}^{n}$  such that
$$
\|f\|_{L^p(\mathbb{R}^n)}:=\left(\int_{\mathbb{R}^n}|f(x)|^p \mathrm{~d} x\right)^{\frac{1}{p}}<\infty.
$$

Now, we give the definition of Lorentz spaces.

\begin{definition}\label{de1}\cite{L} Given a measurable function $f$ on $\mathbb{R}^n$ and $0<p<\infty, 0<$ $r < \infty$, define
$$
\|f\|_{L^{p, r}\left(\mathbb{R}^n\right)}:= \left(\int_0^{\infty}\left(t^{\frac{1}{p}} f^*(t)\right)^r \frac{dt}{t}\right)^{\frac{1}{r}}.$$
Then the Lorentz space $L^{p, r}(\mathbb{R}^n)$ is the set of all functions $f$ such that $\|f\|_{L^{p, r}(\mathbb{R}^n)}<\infty$.

If we take $r=p$, then the Lorentz space $L^{p, r}(\mathbb{R}^n)$ is the Lebesgue space  $L^{p}(\mathbb{R}^{n})$.
\end{definition}

Our first result can be stated as follows.
\begin{theorem}\label{TH1.1} Let $0 \leq \alpha<n$ and $b$ be a locally integrable function. If $1<r<\infty$, $1<p<\frac{n}{\alpha}$ and $p^{\ast}=\frac{pn}{n-\alpha p}$, then the following statements are equivalent:

(1) $b \in BMO(\mathbb{R}^{n})$.

(2) $ M_{\alpha,b}$ is bounded from $L^{p, r}(\mathbb{R}^n)$ to $L^{p^{\ast}, r}(\mathbb{R}^n)$.

(3) There exists a constant $C>0$ such that

\begin{equation} \label{eq1.3}
\sup _{B} \frac{\|(b-b_{B})\chi_{B}\|_{L^{p^{\ast}, r}(\mathbb{R}^n)}}{\|\chi_{B}\|_{L^{p^*, r}(\mathbb{R}^n)}} \leq C. \tag{1.2}
\end{equation}

(4) There exists a constant $C>0$ such that
\begin{equation} \label{eq1.4}
\sup _{B} \frac{\| (b-b_{B})\chi_{B}\|_{L^{1}(\mathbb{R}^n)}}{|B|}  \leq C.\tag{1.3}
\end{equation}
\end{theorem}

Here is the second result we obtained.
\begin{theorem} \label{TH1.2} Let $0 \leq \alpha<n$ and $b$ be a locally integrable function. If $1<r<\infty$, $1<p<\frac{n}{\alpha}$ and $p^{\ast}=\frac{pn}{n-\alpha p}$, then the following statements are equivalent:

(1) $b \in BMO(\mathbb{R}^{n})$ and $b^{-} \in L^{\infty}(\mathbb{R}^n)$.

(2) $[b, M_{\alpha}]$ is bounded from $L^{p, r}(\mathbb{R}^n)$ to $L^{p^{\ast}, r}(\mathbb{R}^n)$.

(3) There exists a constant $C>0$ such that
\begin{equation} \label{eq1.1}
\sup _{B} \frac{\|\left(b-M_{B}(b)\right)\chi_{B}\|_{L^{p^*, r}(\mathbb{R}^n)}}{\|\chi_{B}\|_{L^{p^*, r}(\mathbb{R}^n)}}\leq C.\tag{1.4}
\end{equation}

(4) There exists a constant $C>0$ such that
\begin{equation} \label{eq1.2}
\sup _{B} \frac{\| \left(b-M_{B}(b)\right)\chi_{B}\|_{L^{1}(\mathbb{R}^n)}}{|B|} \leq C.\tag{1.5}
\end{equation}
\end{theorem}

\section{Preliminaries}

To prove our results, we give some necessary  lemmas in this section.
\begin{lemma} \cite{G} \label{le2.1}
Let $0< \alpha <  n$  and $b \in BMO(\mathbb{R}^{n})$.  Then there exists a constant  C such that

$$M_{b,\alpha}f(x)\leq C\|b\|_{BMO(\mathbb{R}^{n})}(M(M_\alpha f)(x)+M_\alpha(Mf)(x))$$
holds for almost every $x\in \mathbb{R}^{n}$ and any locally integrable function $f$.

\end{lemma}

\begin{lemma}\cite{D}\label{le2.4} Let $1<r,p<\infty$  and $B$ be a ball in $\mathbb{R}^{n}$. Then
$$
\|\chi_{B}\|_{L^{p, r}\left(\mathbb{R}^n\right)} \approx |B|^{\frac{1}{p}}.
$$
\end{lemma}
\begin{lemma}\cite{DK}\label{le2.5} Let $1<q, q', r, r'<\infty$, $\frac{1}{q}+\frac{1}{q'}=1$ and $\frac{1}{r}+\frac{1}{r'}=1$. Suppose that $f \in L^{q, r}(\mathbb{R}^n)$ and $g \in L^{q', r'}(\mathbb{R}^n)$. Then
$$
\|f g\|_{L^{1}(\mathbb{R}^n)} \leq 2\|f\|_{L^{q, r}(\mathbb{R}^n)}\|g\|_{L^{q', r'}(\mathbb{R}^n)}.
$$

\end{lemma}

\begin{lemma}\cite{O}\label{le2.2} Let $0 \leq \alpha<n$, $1<r<\infty$, $1<p<\frac{n}{\alpha}$ and $p^{\ast}=\frac{pn}{n-\alpha p}$.
If  $f\in L^{p, r}(\mathbb{R}^n)$, then
$$\|M_{\alpha} f\|_{L^{p^{\ast}, r}(\mathbb{R}^n)} \leq C\|f\|_{L^{p, r}(\mathbb{R}^n)},$$
where the constant $C$ is independent of $f$.
\end{lemma}

\begin{lemma}\cite{ZW}\label{le2.0}
Let  $0 \leq \alpha<n$, $B$  be a ball in  $\mathbb{R}^{n}$  and  $f$  be a locally integrable function. Then, for any  $x \in B$,
$$M_{\alpha}\left(f \chi_{B}\right)(x)=M_{\alpha, B}(f)(x). $$
\end{lemma}

\begin{lemma}\cite{BMR}\label{le2.6}
For any fixed ball $B$, let $E=\{x \in B: b(x) \leq b_B\}$ and $F=\{x \in B: b(x)>b_B\}$. Then
$$
\int_E|b(x)-b_B|dx=\int_F|b(x)-b_B| dx.
$$
\end{lemma}

\section{Proofs of Theorems \ref{TH1.1}-\ref{TH1.2}}

\begin{proof}[Proof of Theorem \ref{TH1.1}] (1) $\Rightarrow$ (2): Assume $b \in BMO(\mathbb{R}^{n})$. By Lemma \ref{le2.1} and Lemma \ref{le2.2}, we have
$$
\begin{aligned}
\|M_{\alpha,b}(f)\|_{L^{p^{\ast}, r}(\mathbb{R}^n)} & \leq C\|b\|_{BMO(\mathbb{R}^{n})}\|(M(M_\alpha f)(x)+M_\alpha(Mf)(x))\|_{L^{p^{\ast}, r}(\mathbb{R}^n)} \\
& \lesssim\|M_\alpha f\|_{L^{p^{\ast}, r}(\mathbb{R}^n)}+\|Mf\|_{L^{p, r}(\mathbb{R}^n)}\\
& \lesssim\|f\|_{L^{p, r}(\mathbb{R}^n)}.
\end{aligned}
$$
Thus, we obtain that $M_{\alpha,b}$ is bounded from $L^{p, r}(\mathbb{R}^n)$ to $L^{p^{\ast}, r}(\mathbb{R}^n)$.

$(2) \Rightarrow(3)$: For any fixed ball $B\subset \mathbb{R}^n$ and $x \in B$, we have
$$
\begin{aligned}
|b(x)-b_{B}| & \leq \frac{1}{|B|} \int_{B}|b(x)-b(y)| d y \\
& =\frac{1}{|B|^{\frac{\alpha}{n}}} \frac{1}{|B|^{1-\frac{\alpha}{n}}} \int_{B}|b(x)-b(y)| \chi_{B}(y) d y \\
& \leq|B|^{-\frac{\alpha}{n}} M_{\alpha, b}(\chi_{B})(x) .
\end{aligned}
$$
Since $M_{\alpha,b}$ is bounded from $L^{p, r}(\mathbb{R}^n)$ to $L^{p^{\ast}, r}(\mathbb{R}^n)$, then by Lemma \ref{le2.4} and noting that $p^{\ast}=\frac{pn}{n-\alpha p}$, we obtain
$$
\begin{aligned}
 \frac{\|(b-b_{B})\chi_{B}\|_{L^{p^*, r}(\mathbb{R}^n))}}{\|\chi_{B}\|_{L^{p^*, r}(\mathbb{R}^n)}}
& \leq\frac{1}{|B|^{\frac{\alpha}{n}}} \frac{\|M_{\alpha, b}(\chi_{B})\|_{L^{p^*, r}(\mathbb{R}^n)}}{\|\chi_{B}\|_{L^{p^*, r}(\mathbb{R}^n)}}
 \\
&  \lesssim \frac{1}{|B|^{\frac{\alpha}{n}}} \frac{\|\chi_{B}\|_{L^{p, r}(\mathbb{R}^n))}}{\|\chi_{B}\|_{L^{p^*, r}(\mathbb{R}^n)}}
 \\
& \leq C,
\end{aligned}
$$
which implies that (\ref{eq1.3}) holds since the ball $B\subset\mathbb{R}^{n}$ is arbitrary.

(3) $\Rightarrow$ (4): Assume (\ref{eq1.3}) holds, we will prove (\ref{eq1.4}). For any fixed ball $B$, by Lemma \ref{le2.4} and Lemma \ref{le2.5}, it is easy to see
$$
\begin{aligned}
\frac{1}{|B|} \int_{B}|b(x)-b_{B}| d x
& \lesssim \frac{1}{|B|} \|(b-b_{B})\chi_{B}\|_{L^{{p^*}, r}(\mathbb{R}^n)} \|\chi_{B}\|_{L^{{p^*}', r'}(\mathbb{R}^n)}\\
& \lesssim \frac{\|(b-b_{B})\chi_{B}\|_{L^{p^{\ast}, r}(\mathbb{R}^n)}}{\|\chi_{B}\|_{L^{p^*, r}(\mathbb{R}^n)}}\\
& \leq C.
\end{aligned}
$$

(4) $\Rightarrow$ (1): For any fixed ball $B$, we have
$$
\begin{aligned}
\frac{1}{|B|} \int_{B}|b(x)-b_{B}| d x
&=\frac{\| (b-b_{B})\chi_{B}\|_{L^{1}(\mathbb{R}^n)}}{|B|}\\
&\leq\sup _{B}\frac{\| (b-b_{B})\chi_{B}\|_{L^{1}(\mathbb{R}^n)}}{|B|}\\
&\leq C.
\end{aligned}
$$
which implies that $b \in BMO(\mathbb{R}^{n})$.

The proof of Theorem \ref{TH1.1} is  finished.
\end{proof}

\begin{proof}[Proof of Theorem \ref{TH1.2}]
(1) $\Rightarrow(2)$: Assume $b \in BMO(\mathbb{R}^{n})$ and $b^{-} \in L^{\infty}(\mathbb{R}^n)$. Using Lemma \ref{le2.1}, Lemma \ref{le2.2} and (\ref{5}), we conclude that
\begin{align*}
\|[b, M_{\alpha}](f)\|_{L^{p^{\ast}, r}(\mathbb{R}^n)}&\leq\|M_{b,\alpha}(f)+2b^{-}M_{\alpha}(f)\|_{L^{p^{\ast}, r}(\mathbb{R}^n)}\\
&\leq \|M_{b,\alpha}(f)\|_{L^{p^{\ast}, r}(\mathbb{R}^n)}+\|2b^{-}M_{\alpha}(f)\|_{L^{p^{\ast}, r}(\mathbb{R}^n)}\\
&\lesssim \|b\|_{BMO(\mathbb{R}^{n})}\|f\|_{L^{p, r}(\mathbb{R}^n)}+\|b^{-}\|_{L^{\infty}(\mathbb{R}^n)}\|f\|_{L^{p, r}(\mathbb{R}^n)}\\
&\lesssim \|f\|_{L^{p, r}(\mathbb{R}^n)}.
\end{align*}
Thus, we obtain that $[b,  M_{\alpha}]$ is bounded from $L^{p, r}(\mathbb{R}^n)$ to $L^{p^{\ast}, r}(\mathbb{R}^n)$.

$(2) \Rightarrow(3)$: We divide the proof into two cases based on the scope of $\alpha$.

{\bf Case 1.} Assume $0 <\alpha<n$. For any fixed ball B,
$$
\begin{aligned}
 \frac{\|(b-M_{B}(b))\chi_{B}\|_{L^{p^*, r}(\mathbb{R}^n)}}{\|\chi_{B}\|_{L^{p^*, r}(\mathbb{R}^n)}}
& \leq\frac{\|(b-|B|^{-\frac{\alpha}{n}} M_{\alpha, B}(b))\chi_{B}\|_{L^{p^*, r}(\mathbb{R}^n)}}{\|\chi_{B}\|_{L^{p^*, r}(\mathbb{R}^n)}}  \\
& \quad+\frac{\|(|B|^{-\frac{\alpha}{n}} M_{\alpha, B}(b)-M_{B}(b))\chi_{B}\|_{L^{p^*, r}(\mathbb{R}^n)} }{\|\chi_{B}\|_{L^{p^*, r}(\mathbb{R}^n)}} \\
& :=I+II .
\end{aligned}
$$
For $I$. By the definition of $M_{\alpha, B}$, we can see that, for any $x \in B$,
\begin{equation}\label{eq3.1}
M_{\alpha, B}(\chi_{B})(x)=|B|^{\frac{\alpha}{n}}. \tag{3.1}
\end{equation}
Using Lemma \ref{le2.0}, for any $x \in B$, we have
$$
\begin{aligned}
M_{\alpha}(\chi_{B})(x)=M_{\alpha, B}(\chi_{B})(x)=|B|^{\frac{\alpha}{n}},\quad
M_{\alpha}(b \chi_{B})(x)=M_{\alpha, B}(b)(x).
\end{aligned}
$$
Thus, for any $x \in B$,
$$
\begin{aligned}
b(x)-|B|^{-\frac{\alpha}{n}} M_{\alpha, B}(b)(x) & =|B|^{-\frac{\alpha}{n}}(b(x)|B|^{\frac{\alpha}{n}}-M_{\alpha, B}(b)(x)) \\
& =|B|^{-\frac{\alpha}{n}}(b(x) M_{\alpha}(\chi_{B})(x)-M_{\alpha}(b \chi_{B})(x)) \\
& =|B|^{-\frac{\alpha}{n}}[b, M_{\alpha}](\chi_{B})(x) .
\end{aligned}
$$
Since $\left[b, M_{\alpha}\right]$ is bounded from $L^{p, r}(\mathbb{R}^n)$ to $L^{p^{\ast}, r}(\mathbb{R}^n)$, then by Lemma \ref{le2.4} and noting that $p^{\ast}=\frac{pn}{n-\alpha p}$, we have
$$
\begin{aligned}
I & =\frac{\|(b-|B|^{-\frac{\alpha}{n}} M_{\alpha, B}(b))\chi_{B}\|_{L^{p^*, r}(\mathbb{R}^n)}}{\|\chi_{B}\|_{L^{p^*, r}(\mathbb{R}^n)}}\\
& =\frac{1}{|B|^{\frac{\alpha}{n}}} \frac{\|[b, M_{\alpha}](\chi_{B})\|_{L^{p^*, r}(\mathbb{R}^n)}}{\|\chi_{B}\|_{L^{p^*, r}(\mathbb{R}^n)}} \\
& \lesssim \frac{1}{|B|^{\frac{\alpha}{n}}} \frac{\|\chi_{B}\|_{L^{p, r}(\mathbb{R}^n)}}{\|\chi_{B}\|_{L^{p^*, r}(\mathbb{R}^n)}} \\
& \leq C.
\end{aligned}
$$
Next, we estimate $II$. Similar to (\ref{eq3.1}), by Lemma \ref{le2.0} and noting that, for any  $x \in B$,
$$
M_{B}\left(\chi_{B}\right)(x)=\chi_{B}(x),
$$
it is easy to see that, for any $x \in B$,
\begin{equation}\label{eq3.2}
M\left(\chi_{B}\right)(x)=\chi_{B}(x), \quad M\left(b \chi_{B}\right)(x)=M_{B}(b)(x).\tag{3.2}
\end{equation}
Then, by (\ref{eq3.1}) and (\ref{eq3.2}), for any $x \in B$, we get
$$
\begin{aligned}
\left|| B|^{-\frac{\alpha}{n}} M_{\alpha, B}(b)(x)-M_{B}(b)(x) \right|
&\leq  |B|^{-\frac{\alpha}{n}}\left|M_{\alpha}(b \chi_{B})(x)-|b(x)|M_{\alpha}(\chi_{B})(x)\right| \\
&\quad+|B|^{-\frac{\alpha}{n}}\left|| b(x)|M_{\alpha}(\chi_{B})(x)-M_{\alpha}(\chi_{B})(x) M(b \chi_{B})(x)\right| \\
&= |B|^{-\frac{\alpha}{n}}\left|M_{\alpha}(|b| \chi_{B})(x)-| b(x)|M_{\alpha}(\chi_{B})(x)\right| \\
&\quad+|B|^{-\frac{\alpha}{n}} M_{\alpha}\left(\chi_{B}\right)(x)\left|| b(x)|M\left(\chi_{B}\right)(x)-M\left(b \chi_{B}\right)(x)\right| \\
&= |B|^{-\frac{\alpha}{n}}\left|[|b|, M_{\alpha}](\chi_{B})(x)\right|+\left|[|b|, M](\chi_{B})(x)\right| .
\end{aligned}
$$
Since $\left[b, M_{\alpha}\right]$ is bounded from $L^{p, r}(\mathbb{R}^n)$ to $L^{p^{\ast}, r}(\mathbb{R}^n)$.
Then, by Lemma \ref{le2.4}, we have
$$
\begin{aligned}
II & \leq\frac{\left\|\Big(|B|^{-\frac{\alpha}{n}}\big|[|b|, M_{\alpha}](\chi_{B})\big|+\big|[|b|, M](\chi_{B})\big|\Big)\chi_{B}\right\|_{L^{p^*, r}(\mathbb{R}^n)} }{\|\chi_{B}\|_{L^{p^*, r}(\mathbb{R}^n)}}\\
& \lesssim \frac{1}{|B|^{\frac{\alpha}{n}}}\frac{\|\chi_{B}\|_{L^{p, r}(\mathbb{R}^n)} }{\|\chi_{B}\|_{L^{p^*, r}(\mathbb{R}^n)}}+\frac{\|\chi_{B}\|_{L^{p^*, r}(\mathbb{R}^n)} }{\|\chi_{B}\|_{L^{p^*, r}(\mathbb{R}^n)}} \\
& \leq C.
\end{aligned}
$$
This gives the desired estimate
$$
\frac{\|\left(b-M_{B}(b)\right)\chi_{B}\|_{L^{p^*, r}(\mathbb{R}^n)}}{\|\chi_{B}\|_{L^{p^*, r}(\mathbb{R}^n)}}\leq C,
$$
which implies that (\ref{eq1.1}) holds.

{\bf Case 2.} Assume $\alpha=0$. For any fixed ball $B$ and $x \in B$, by (\ref{eq3.2}), we can get
$$
b(x)-M_{B}(b)(x)=b(x) M(\chi_{B})(x)-M(b \chi_{B})(x)=[b, M](\chi_{B})(x).
$$
Assume that $[b, M]$ is bounded from $L^{p^{*}, r}(\mathbb{R}^n)$ to $L^{p^{\ast}, r}(\mathbb{R}^n)$, then by Lemma \ref{le2.4}, we have
$$
\begin{aligned}
\frac{\|(b-M_{B}(b))\chi_{B}\|_{L^{p^*, r}(\mathbb{R}^n)}}{\|\chi_{B}\|_{L^{p^*, r}(\mathbb{R}^n)}}
& = \frac{\|[b, M](\chi_{B})\|_{L^{p^*, r}(\mathbb{R}^n)}}{\|\chi_{B}\|_{L^{p^*, r}(\mathbb{R}^n)}} \\
& \lesssim\frac{\|\chi_{B}\|_{L^{p^*, r}(\mathbb{R}^n)}}{\|\chi_{B}\|_{L^{p^*, r}(\mathbb{R}^n)}}  \\
& \leq C,
\end{aligned}
$$
which deduces that (\ref{eq1.1}).

(3) $\Rightarrow$ (4): Assume (\ref{eq1.1}) holds, then for any fixed ball $B$, by Lemma \ref{le2.5}, we conclude that
$$
\begin{aligned}
\frac{1}{|B|} \int_{B}\left|b(x)-M_{B}(b)(x)\right| d x
&\lesssim \frac{1}{|B|} \left\|(b-M_{B}(b))\chi_{B}\right\|_{L^{{p^*}, r}(\mathbb{R}^n)} \|\chi_{B}\|_{L^{{p^*}', r'}(\mathbb{R}^n)}\\
&\lesssim \frac{1}{|B|} \frac{\|(b-M_{B}(b))\chi_{B}\|_{L^{p^{\ast}, r}(\mathbb{R}^n)}}{\|\chi_{B}\|_{L^{p^*, r}(\mathbb{R}^n)}}
\\
&\leq C,
\end{aligned}
$$
where the constant $C$ is independent of $B$. Thus we have (\ref{eq1.2}).

(4) $\Rightarrow$ (1): To prove $b \in BMO(\mathbb{R}^{n})$, it suffices to show that there is a constant $C>0$ such that, for any fixed ball $B$,
$$
\frac{1}{|B|} \int_B|b(x)-b_B| dx \leq C .
$$
For any fixed ball $B$, let $E=\left\{x \in B: b(x) \leq b_B\right\}$ and $F=\left\{x \in B: b(x)>b_B\right\}$.
Since for any $x \in E$, we have $b(x) \leq b_B \leq M_B(b)(x)$, then
\begin{equation}\label{eq3.3}
|b(x)-b_B| \leq|b(x)-M_B(b)(x)| .\tag{3.3}
\end{equation}
By Lemma \ref{le2.6} and (\ref{eq3.3}), we get
$$
\begin{aligned}
\frac{1}{|B|} \int_B|b(x)-b_B| d x
& =\frac{2}{|B|} \int_E|b(x)-b_B| d x \\
& \leq \frac{2}{|B|} \int_E|b(x)-M_B(b)(x)| d x \\
& \leq \frac{2}{|B|} \int_B|b(x)-M_B(b)(x)| d x\\
& \leq C.
\end{aligned}
$$
Thus, we deduce that $b \in BMO(\mathbb{R}^{n})$.

Next, let us show that $b^{-} \in L^{\infty}(\mathbb{R}^n)$. Observe that $0 \leq b^{+}(y) \leq|b(y)| \leq$ $M_B(b)(y)$ for any $y \in B$, therefore
$$
0 \leq b^{-}(y) \leq M_B(b)(y)-b^{+}(y)+b^{-}(y)=M_B(b)(y)-b(y) .
$$
Then, for any ball $B$, we have
$$
\begin{aligned}
\frac{1}{|B|} \int_B b^{-}(y) d y & \leq \frac{1}{|B|} \int_B\left(M_B(b)(y)-b(y)\right) d y \\
& =\frac{1}{|B|} \int_B\left|b(y)-M_B(b)(y)\right| d y\\
&\leq C .
\end{aligned}
$$
Let $|B| \rightarrow 0$ with $x \in B$. Lebesgue's differentiation theorem implies that
$$
0 \leq b^{-}(x)=\lim _{|B| \rightarrow 0} \frac{1}{|B|} \int_B b^{-}(y) d y \leq C .
$$
Thus, we conclude that $b^{-} \in L^{\infty}(\mathbb{R}^n)$.

This completes the proof of Theorem \ref{TH1.2}.
\end{proof}

\end{document}